\documentclass[svgnames]{amsart}
\usepackage[english]{babel}
\usepackage[utf8]{inputenc}

\usepackage[margin=3cm]{geometry}

\usepackage[colorlinks=true, pdfstartview=FitV, linkcolor=black, citecolor=black, urlcolor=black]{hyperref}
\usepackage[nameinlink]{cleveref}

\usepackage{amssymb, amsfonts, amsmath, amsthm, amscd, amsxtra, latexsym, mathabx, mathtools, enumitem}
\usepackage{tikz}\usetikzlibrary{cd}
\usepackage{epigraph}

\theoremstyle{plain}
\newtheorem{thm}{Theorem}[section] 
\newtheorem{lem}[thm]{Lemma} 
\AddToHook{env/lem/begin}{\crefalias{thm}{lem}}

\AddToHook{env/claim/begin}{\crefalias{thm}{claim}}

\newtheorem{prop}[thm]{Proposition}
\AddToHook{env/prop/begin}{\crefalias{thm}{prop}}
\newtheorem{cor}[thm]{Corollary}
\AddToHook{env/cor/begin}{\crefalias{thm}{cor}}

\newtheorem{thmintro}{Theorem}

\AddToHook{env/corintro/begin}{\crefalias{thmintro}{corintro}}
\newtheorem{questintro}[thmintro]{Question}
\AddToHook{env/questintro/begin}{\crefalias{thmintro}{questintro}}

\AddToHook{env/remintro/begin}{\crefalias{thmintro}{remintro}}

\theoremstyle{definition}
\newtheorem{defn}[thm]{Definition}
\AddToHook{env/defn/begin}{\crefalias{thm}{defn}}

\AddToHook{env/notation/begin}{\crefalias{thm}{notation}}
\newtheorem{rem}[thm]{Remark}
\AddToHook{env/rem/begin}{\crefalias{thm}{rem}}
\newtheorem{ex}[thm]{Example}
\AddToHook{env/ex/begin}{\crefalias{thm}{ex}}
\newtheorem{construction}[thm]{Construction}
\AddToHook{env/construction/begin}{\crefalias{thm}{construction}}

\newcommand{\Cay}[2]{\operatorname{Cay}\left(#1,#2\right)}
\newcommand{\dom}[1]{\mathbf{dom}\left(#1\right)}
\newcommand{\ran}[1]{\mathbf{ran}\left(#1\right)}

\newcommand{\Z}{\mathbb{Z}}

\newcounter{gcomments}

\newcounter{fcomments}

\title[Milnor-Schwarz for inverse monoids]{Geometric Inverse Semigroup Theory: a note on the Milnor-Schwarz Lemma for inverse monoids}

\author[G. Mangioni]{Giorgio Mangioni}
    \address{(Giorgio Mangioni) Maxwell Institute and Department of Mathematics, Heriot-Watt University, Edinburgh, UK}
    \email{gm2070@hw.ac.uk}

\author[F. Tesolin]{Francesco Tesolin}
    \address{(Francesco Tesolin) Maxwell Institute and Department of Mathematics, Heriot-Watt University, Edinburgh, UK}
    \email{ft2021@hw.ac.uk}

\begin{document}
\begin{abstract} We generalise the Milnor-Schwarz lemma to inverse monoids acting on presheaves of geodesic metric spaces. We provide two proofs of this fact: one only uses elementary techniques, inspired by the arguments for group actions on metric spaces; the other involves a version of the Vietoris-Rips complex, and builds on work of Chung-Martínez-Szakács. 
\end{abstract}
\maketitle

\epigraph{I used to only do GGT, but now I'm getting the GIST of this.}{G. M.}

\small{MSC subject classification: 20F65, 20M18}
\section*{Introduction}
Where groups can be thought of as sets of symmetries, an inverse semigroup $S$ is a set of partial symmetries, that is, each element has a ``local"
inverse and their composition produces an \emph{idempotent} element (that is, one such that $ee=e$). The set of all idempotents $E(S)$ forms a meet-semilattice, with the meet operation corresponding to multiplication.
The fundamental example of an inverse semigroup is the collection of partial bijections on a set, with each idempotent being the identity on some subset; see \Cref{example: partial} for details.

There has been a recent growing interest in geometric inverse semigroup theory \cite{gray2013groups,magan2021amenability, gray2022algorithmic,Chung_Martinez_szakacs}, where the authors study inverse semigroups as extended metric spaces. For example, given a semigroup $S$, together with a \emph{quasi-generating set} $\mathcal M\subseteq S$ (that is, a subset which generates $S$ together with $E(S)$), one can equip $S$ with the \emph{Cayley metric} $d_{\mathcal{M}}$, which is the direct analogue of the Cayley metric (word metric) for groups. However, up to this point, there has been
no ``proper" analogue of the renowned \emph{Milnor-Schwarz lemma} from geometric group theory, which allows one to study the geometry of a group via the geometry of a metric space on which the group acts ``nicely" \cite{Schwarz, Milnor}. 

To fill this gap, we first need the correct notion of action of an inverse semigroup, which is that of an \emph{étale action} introduced in \cite{Steinberg_morita_eq} and studied further in \cite{tesolin2025generalisationmunnsemigroup}. Namely, given a semigroup $S$, with idempotents $E(S)$, we consider an action of $S$ on a \emph{presheaf} over $E(S)$: see \Cref{defn:presheaf} for further details. This approach of acting on a presheaf is an effective tool to generalise ideas from group theory, as it recovers the usual definition of group action when $S$ is a group, which is precisely an inverse semigroup whose unique idempotent is the identity element.

In order to do geometric inverse semigroup theory, we then restrict to (right) actions by $1$-Lipschitz maps on presheaves of geodesic metric spaces (see \Cref{defn:action}). For reasons that we clarify below, we also assume the inverse semigroup admits an identity element, which makes it an \emph{inverse monoid}. In this framework there are natural generalisations of the notions of properness and coboundedness, which restrict to the usual notions in the group setting; we defer to \Cref{defn:proper_and_cobounded} below for details. We now state our main theorem, which recovers the usual Milnor-Schwarz lemma for groups:
\begin{thmintro}\label{thmintro}
	Let $S$ be an inverse monoid, with identity element $1$, and $(X, E(S),p)$ be a presheaf of geodesic metric spaces. If $S$ acts on $X$ properly and coboundedly then:
    \begin{enumerate}
        \item\label{item_fingen} There exists a finite, quasi-generating set $\mathcal{M}\subseteq S$.
        \item There exists $x_1\in X$ such that $p(x_1)=1$ and, if $d_{\mathcal{M}}$ is the Cayley metric associated to $\mathcal{M}$, the orbit map $(S,d_{\mathcal{M}})\to X$ sending $s\in S$ to $x_1\cdot s$ is an order-preserving quasi-isometry.
    \end{enumerate}
\end{thmintro}

\noindent Let us comment on the various requirements of the above theorem. 

\textit{Inverse monoid.} As pointed out in \Cref{rem:need_1_for_qi} below, if $S$ is an inverse semigroup and there exists an orbit map $S\to X$ defining a quasi-isometry, then $S$ must have an identity element, and therefore be an inverse monoid. Regarding \eqref{item_fingen}, one could restrict to the case where $S$ admits a finite, non-trivial collection of maximal idempotent elements, as is the case for the free inverse semigroup on a finite alphabet. In this setting, we expect one could find suitable generalisations of the notions of properness and coboundedness which would imply finite quasi-generation. However, our proof of \Cref{thm:MilnorSchwarz} relies on the existence of an orbit map, which would not be clearly defined if one had to take basepoints in multiple fibers of $X$; hence, extra effort would be required to understand the connection between the geometry of the inverse semigroup and that of the presheaf.

\textit{Geodesic metric spaces.} In analogy with the group case, our theorem would work if one only required the metric spaces to be \emph{quasi-geodesic}, meaning that there are uniform constants $K,C$ such that every two points in the same metric space are connected by a $(K,C)$-quasi-geodesic. The proof would follow along the same lines of the more natural geodesic case, so we decided to stick to the latter for simplicity. However, the theorem does not hold if one does not require the fibers to be ``coarse length spaces" in some sense: for instance, $\Z$ acts properly and transitively on itself with the metric $d(m,n)=\sqrt{|m-n|}$, but any Cayley metric on $\Z$ is quasi-isometric to $\mathbb{R}$. 

\textit{$1$-Lipschitz maps.} One cannot expect the action to be by genuine isometries, but only by $1$-Lipschitz maps: for example, if a semigroup admits a zero element, then multiplying by it contracts the whole Cayley graph to a point. 

\textit{Properness.} As in the group case, the definition of properness is tailored to get finite quasi-generation in \Cref{thmintro}.
If one drops this requirement, one still gets a possibly infinite quasi-generating set $\mathcal M$ whose Cayley metric is quasi-isometric to the presheaf: see \Cref{rem:no_properness_is_fine} below.

\textit{Quasi-finite generation.} Finally, note that, in general, one cannot hope to get a genuinely finite generating set. This is because any meet-semilattice $E$ with a top element is itself an inverse monoid, and if $E$ is uncountable then it cannot be finitely generated (see \Cref{rem:nofinitegen}).


\subsection*{Outline of the paper}
In \Cref{sec:background} we give the background on inverse semigroups and how they act on presheaves of metric spaces. In \Cref{defn: cayley graph}
we reinterpret the Cayley graph of an inverse semigroup as a presheaf of geodesic metric spaces (the Sch\"{u}tzenberger graphs).

In the two following sections, we give two proofs of \Cref{thmintro}. The first is
an elementary argument adapting the proof of the usual Milnor-Schwarz lemma for groups. 
The second uses a theorem of Chung-Martínez-Szakács \cite{Chung_Martinez_szakacs}, which roughly states that any two ``Cayley-like" metrics on an inverse semigroup are coarsely equivalent. In order to produce such a metric, we introduce an analogue of the \emph{Vietoris-Rips complex} for group actions, which we believe to be of independent interest. 


\subsection*{Further questions}
This work provides further evidence that presheaves (of metric spaces) are the natural objects on which inverse semigroups act. 
Using our generalisation of the Milnor-Schwarz lemma, we hope that more examples of quasi-finitely generated 
inverse monoids with interesting geometric properties may be produced. In this direction, \cite{gray2022algorithmic} considered the properties of finitely generated inverse monoids with hyperbolic Sch\"{u}tzenberger graphs, so we ask the following vague question:

\begin{questintro}
What are the properties of a (finitely generated?) inverse monoid that acts properly and coboundedly on a presheaf of CAT(0) (resp. 
systolic, injective, median) spaces?
\end{questintro}

\noindent The main results of \cite{gray2022algorithmic} show that, while finitely presented inverse monoids whose Sch\"{u}tzenberger graphs are quasitrees enjoy several nice algorithmic and language-theoretic properties, just requiring hyperbolicity of the Sch\"{u}tzenberger graphs is not even enough to get a decidable word problem. However, besides the metric spaces appearing as fibers, we expect that the presheaf structure itself might provide another element of flexibility in studying the geometry of inverse semigroups. In other words, it might be possible to infer better properties of an inverse monoid by looking at actions where not only the fibers are restricted to specific geometries (for example hyperbolicity), but also the restriction maps of the presheaf satisfy suitable properties.

\subsection*{Acknowledgements}
We thank our supervisors Mark Lawson and Alessandro Sisto for sharing their expertise with us. We are grateful to the anonymous referee for their helpful suggestions in the preparation of this work.

\section{Background on inverse semigroups}\label{sec:background}

\begin{defn}[Inverse semigroup, domain and range]
An \emph{inverse semigroup} is a semigroup $S$ such that for all $s \in S$ there exists a unique $s^{-1} \in S$ such that $ss^{-1}s = s$ and $s^{-1}ss^{-1} = s^{-1}$. An \emph{inverse monoid} is an inverse semigroup $S$ with an element $1 \in S$ such that $1s= s1= s$ for all $s \in S$. For each element $s \in S$ we call the element $s^{-1}s$ the \emph{domain} of $s$ and denote it by $\dom{s}$, and call $ss^{-1} $ the \emph{range} of $s$ and denote it by $\ran{s}$.
\end{defn}

\begin{defn}[Idempotents]
	An \emph{idempotent} of an inverse semigroup $S$ is an element $e \in S$ such that $ee=e$. The set of all idempotents in $S$ is denoted by $E(S)$. Notice that, for every $s\in S$, both $\dom{s}$ and $\ran{s}$ are idempotents by construction. Furthermore, an immediate corollary of the definition, as shown in \cite{LawsonBook}, is that any two idempotents commute.
\end{defn}

\begin{defn}[Partial order]
	Let $S$ be an inverse semigroup. Set $s \leq t$ if
    and only if $s = et$ for some idempotent $e \in E(S)$. As
    shown in \cite{LawsonBook}, this is a partial order, termed the \emph{natural partial order} of $S$.
    With respect to the natural partial order, the idempotents $E(S)$
    form a meet-semilattice, with meet operation $e\wedge f=ef$.
\end{defn}

\noindent Note that a group is precisely an inverse semigroup with a unique idempotent element. More generally:
\begin{lem}\label{lem:unique_max_implies_monoid}
    If $E(S)$ admits a unique maximal element $e$, then $S$ is an inverse monoid with $e$ the identity element.
\end{lem}
\begin{proof}
    For every $s\in S$, we have that
    \[es=e\ran{s}s=(e\wedge \ran{s})s=\ran{s}s=s,\]
    where we used that $\ran{s}\le e$ by maximality. Using the domain instead of the range shows that $se=s$, and therefore $e$ is the identity element.
\end{proof}

\begin{ex}[Partial bijections] \label{example: partial}
	Let $X$ be a set, and let $S$ be the collection of all partial bijections, that is, bijections between subsets of $X$. $S$ is an inverse monoid with the following multiplication: given $A,B,C,D$ subsets of $X$ and partial bijections $f\colon A\to B$ and $g\colon C\to D$, the product $gf$ is the map $(g\circ f)|_{f^{-1}(B\cap C)} \colon  f^{-1}(B\cap C)\to g(B\cap C)$.  Furthermore, given $f\colon A\to B$, the domain and range of $f$ are the identity maps on $A$ and $B$, respectively. If $e\in S$ is an idempotent, then $e$ is the identity on some subset $A\subseteq X$; hence, given $f,g\in S$, $f\le g$ if and only if $f$ is a restriction of $g$.
\end{ex}

\begin{rem}\label{rem:order_under_products}
	Given $s,t\in S$, we have that $\dom{st}\le \dom{t}$, since
	$$\dom{st}\dom{t}=t^{-1}s^{-1}stt^{-1}t=t^{-1}s^{-1}st=\dom{st}.$$
\end{rem}

\begin{defn}[Congruence and Green's relations]
    Let $S$ be an inverse semigroup. A \emph{congruence} on $S$ is an equivalence relation
    $\sim$ such that for all $s,t,a,b \in S$ with $s \sim t$, $a \sim b$ then $sa \sim tb$.
    Two important congruences are
	$\mathcal{L}$ and $\mathcal{R}$, given  by
	\[
		(s,t) \in \mathcal{L} \iff \dom{s}= \dom{t} \mbox{ and } (s,t) \in \mathcal{R} \iff \ran{s}=\ran{t}.
	\]
    It is shown in \cite{LawsonBook} that
    \[
		(a,b) \in \mathcal{L} \text{ if and only if } Sa = Sb
	.\]
\end{defn}

\begin{lem}\label{lem:edge_in_pairs}
   Let $s,t,x\in S$. If $(s,t)\in \mathcal{L}$ and $s=xt$ then $t=x^{-1}s$. Furthermore, if $x\in E(S)$ then $t=s$.
\end{lem}

\begin{proof}
    Since $(s,t)\in \mathcal{L}$ we have that 
    $$\dom{t}=t^{-1}t=t^{-1}x^{-1}xt=\dom{s}.$$
    If we multiply by $t$ on the left we get
    $$t=tt^{-1}x^{-1}xt=x^{-1}xtt^{-1}t=x^{-1}xt=x^{-1}s,$$
    where we used that the idempotents $tt^{-1}$ and $x^{-1}x$ commute. 
    
    For the ``furthermore" part, if $x\in E(S)$ then $x^{-1}=x$, so
    $$t=x^{-1}s=x^{-1}xt=xt=s,$$
    as required.
\end{proof}

\noindent We now recall the definition of a presheaf over a meet-semilattice, dating back to work of
\cite{Goldblatt}. It was first described in \cite{Steinberg_morita_eq} that inverse semigroups act naturally on such
objects; see also \cite{tesolin2025generalisationmunnsemigroup} for further details.

\begin{defn}[Presheaf of metric spaces]\label{defn:presheaf} Let $E$ be a meet-semilattice, and denote the meet of two elements $e,f\in E$ by $ef$.
A \emph{presheaf (of sets)} over $E$ is a set $X$ equipped with functions $p\colon X \to E$ and $(\cdot)\colon X \times E \to X$ satisfying
\begin{enumerate}
	\item 
	$(x \cdot e) \cdot f = x \cdot ef$;
	\item 
	$x \cdot p(x) = x$;
	\item 
	$p(x \cdot e) = p(x)e$.
\end{enumerate}
Denote such a presheaf by the triple $(X,E,p)$. We call a presheaf \emph{global} when the function $p$ is surjective. 

We denote the \emph{fiber} $p^{-1}(e)$ by $X_e$. For every $e\ge f$, the map $X_e\to X_f$ sending $x$ to $x\cdot f$ is called the \emph{restriction} from $X_e$ to $X_f$. For $x,y \in X$ set $x \leq y$ if $x$ is the restriction of $y$, equivalently $x = y \cdot p(x)$. This defines a partial order on $X$.

A \emph{presheaf of metric spaces} is a global presheaf where each fiber $X_{e}$ is a metric space with metric
$d_{e}\colon X_e^2\to \mathbb{R}_{\ge 0}$ and such that, for all $x,y \in X$ and $e,f \in E$,
 \[
	d_{e}(x,y) \geq d_{ef}(x \cdot f, y \cdot f)
.\]
For the ease of notation, we set $d(x,y)=d_e(x,y)$ if $x,y\in X_e$, and $d(x,y)=\infty$ otherwise. This makes $(X,d)$ an \emph{extended metric space}, i.e. a metric space where distances are allowed to be infinite.

A presheaf of metric spaces $(X,d)$, and more generally an extended metric space, is \emph{geodesic} if, for every $x,y\in X$ such that $D=d(x,y)<\infty$, there exists an isometric embedding of an interval $\gamma\colon [0,D]\to X$ such that $\gamma(0)=x$ and $\gamma(D)=y$.
\end{defn}

\begin{defn}[Acting on a presheaf]\label{defn:action}
Let $S$ be an inverse semigroup and $(X,E(S),p)$ be a presheaf of metric spaces over the idempotents of $S$. We say that $S$ \emph{acts on} $(X,E(S),p)$ if there exists an action $(\cdot)\colon X \times S \to X$ which restricts to the action $X \times E(S) \to X$ and satisfies the following for all $x,y\in X$, $e\in E(S)$ and $s\in S$:
\begin{itemize}
    \item The projection is $S$-equivariant: $p(x \cdot s) = s^{-1}p(x)s$;
    \item The action is by $1$-Lipschitz maps: if $x,y\in X_e$, $d_e(x,y) \geq d_{s^{-1}es}(x \cdot s, y \cdot s) $.
\end{itemize}
\end{defn}

\begin{rem}\label{rem:restriction_is_multiplication}
    Notice that, by definition, every idempotent acts as the restriction to the corresponding fiber.
\end{rem}

\begin{lem} \label{lem:s_induces_isom}
    Let $S$ act on $(X,E(S),p)$ a presheaf of metric spaces. Each $s \in S$ induces an isometry
    $\theta_s \colon X \cdot ss^{-1} \to X \cdot s^{-1}s$ by acting on the right by $s$.
\end{lem}

\begin{proof}
    Let $x_0,y_0\in X$, and let $x=x_0\cdot ss^{-1}$ and $y=y_0\cdot ss^{-1}$. Notice that $p(x) = p(y)$ if and only if
    $p(x \cdot s) = p(y \cdot s)$, and therefore
    $d(x,y) = \infty $ if and only if $d(x \cdot s, y\cdot s) = \infty$.
    Thus suppose that $p(x) = p(y) = e$. By assumption $$e = p(x) = p(x_0 \cdot ss^{-1})
    =p(x_0)ss^{-1};$$ hence
     $e \leq ss^{-1}$ and, since the action is by 1-Lipschitz maps,
    $$d_{e}(x,y) \geq d_{s^{-1}es}(x \cdot s, y\cdot s) \geq d_{e}(x \cdot ss^{-1}, y \cdot ss^{-1}) = d_{e}(x,y).$$
\end{proof}

\begin{rem}
    An action of an inverse semigroup is often defined as a homomorphism onto the partial bijections of a set. The previous lemma shows that an action of $S$ on a global presheaf $(X,E(S),p)$ induces a homomorphism from $S$ onto the partial bijections on $X$ which is \emph{idempotent-separating}, meaning that it is injective on idempotents (this is because, given any $x\in X_e$, then $p(x\cdot e)=p(x)=e$). Actually, by \cite[Theorem D]{tesolin2025generalisationmunnsemigroup} actions of $S$ on a global presheaf (of sets) $(X,E(S),p)$ are in one-to-one correspondence with certain idempotent-separating representations of $S$ as partial bijections of $X$, namely those whose image is a wide subsemigroup of the generalised Munn semigroup $T_X$. 
\end{rem}

\noindent Natural examples of presheaves on which a semigroup acts arise from its \emph{Schützenberger graphs}, first introduced in \cite{stephen1990presentations}. 

\begin{defn}[Cayley graph and Sch\"{u}tzenberger graphs] \label{defn: cayley graph}
    Let $S$ be an inverse semigroup with generating set $\mathcal{G}$. Define
    $\Cay{S}{\mathcal{G}}$ to be the directed graph whose vertices are elements of $S$, 
    and filling in an oriented edge from $s$ to $t$, labelled by some $g \in \mathcal{G}$, if and only if $gs = t$.
    
    Define the \emph{Schützenberger graphs} as the biconnected components of $\Cay{S}{\mathcal{G}}$, that is, the maximal subgraphs where any two vertices $a$ and $b$ are connected by an oriented path from $a$ to $b$ and an oriented path from $b$ to $a$. Notice that Schützenberger graphs coincide with $\mathcal L$-equivalence classes: this is because there is an oriented path from $s$ to $t$ if and only if $xs=t$ for some $x\in S$. Furthermore, within a Schützenberger graph, edges come in inverse pairs by \Cref{lem:edge_in_pairs}. Hence one defines the \emph{Cayley metric} $d_{\mathcal{G}}$ on $S$ by setting 
    $d_{\mathcal{G}}(s,t)=\infty$ if $(s,t)\not\in \mathcal{L}$, and the path distance in their common Schützenberger graph if $(s,t)\in \mathcal{L}$. This makes $(S,d_{\mathcal{G}})$ a presheaf of geodesic metric spaces over $E(S)$, with projection $s\to\dom{s}$ and action $S\times S\to S$ given by right multiplication.
\end{defn}

\noindent As we shall prove in \Cref{lem:cayley_dist_without_idempotents} below, the Cayley metric can actually be computed using only non-idempotent elements, leading us to the definition of a \emph{quasi}-generating set:
\begin{defn}[Quasi-generating set]
    A subset $\mathcal{M}\subseteq S$ is a \emph{quasi-generating set} if $S=\langle \mathcal{M}\cup E(S)\rangle$. We always assume that a quasi-generating set is \emph{symmetric}, meaning that $\mathcal{M}=\mathcal{M}^{-1}\coloneq\{x^{-1}\mid x\in \mathcal{M}\}$.

    We say that $S$ is \emph{quasi-finitely generated} (resp. quasi-countable) if it admits a finite (resp. countable) quasi-generating set. For reference, being quasi-finitely generated is weaker than admitting a \emph{finite labeling} as in \cite[Definition 3.16]{Uniform_Roe_alg}.
\end{defn}

\begin{lem}\label{lem:cayley_dist_without_idempotents}
    Let $\mathcal M$ be a quasi-generating set for $S$. Then $d_{\mathcal M\cup E(S)}=d_{\mathcal{M}}$, where $$d_{\mathcal{M}}(s,t)=\begin{cases}
        \min\{k\mid \exists m_1,\ldots, m_k\in \mathcal{M} \mbox{ s.t. }s= m_1\ldots m_kt\}&\mbox{if }(s,t)\in \mathcal{L};\\
        \infty &\mbox{if }(s,t)\not\in \mathcal{L}.
    \end{cases}$$
\end{lem}

\begin{proof} If $(s,t)\not\in \mathcal {L}$ there is nothing to prove, so suppose otherwise. Write $s= x_1\ldots x_l t$, where each $x_i\in \mathcal M\cup E(S)$. Notice that we can move all $x_i\in \mathcal M$ to the left, because for every idempotent $e$ we have that $ex_i=x_{i}(x_{i}^{-1}ex_{i})$ and the quantity between brackets is still an idempotent. Thus we can write $s= x_{i_1}\ldots x_{i_k} f_1\ldots f_r t$, where each $x_{i_j}$ belongs to $\mathcal M$ and each $f_j$ is an idempotent. Since a product of idempotents is itself idempotent, $f\coloneq f_1\ldots f_r\in E(S)$. By \Cref{rem:order_under_products}, $\dom{s}\le \dom{ft}\le \dom{t}=\dom{s}$, so $(ft,t)\in \mathcal L$. In turn, since $f$ is an idempotent, the furthermore part of \Cref{lem:edge_in_pairs} yields that $ft=t$, so $s= x_{i_1}\ldots x_{i_k}t$, as required. 
\end{proof}
\noindent As in the group case, two finite quasi-generating sets give biLipschitz equivalent Cayley metrics:

\begin{lem}[{\cite[Prop. 5.2.4]{magan2021amenability}}]\label{lem:compare_cayley_metrics}
    Let $S$ be an inverse semigroup, and let $\mathcal{M}$, $\mathcal N$ be two finite quasi-generating sets. Then the identity map on $S$ induces an order-preserving biLipschitz equivalence $(S,d_{\mathcal{M}})\to (S,d_{\mathcal N})$.
\end{lem}

\section{Milnor-Schwarz for inverse monoids}
In this section we present a self-contained proof of the Milnor-Schwarz lemma, which is \Cref{thm:MilnorSchwarz} below. For the next definition, given an extended metric space $(X,d)$, we denote the ball of radius $R$ around a point $x\in X$ by $B(x,R)$.
\begin{defn}\label{defn:proper_and_cobounded}
Let  $S$ be an inverse monoid, acting on a presheaf of metric spaces $(X,E(S),p)$. The action is:
\begin{enumerate}
	\item \emph{proper} if for every $y_1\in X_1$ and every $R\ge 0$, there exists a finite subset $\mathcal{C}=\mathcal{C}(y_1,R)\subseteq S$ such that
	\[
		\{ s \in S \mid d(y_1 \cdot s, y_1\cdot \dom{s})\le R\}\subseteq \mathcal{C} E(S).
	\]
    \item \emph{cobounded} if there exist $x_1\in X_1$ and $T\ge 0$ such that $B(x_1,T) \cdot S = X$. When we want to stress the constant $T$, we say the action is $T$-cobounded.
\end{enumerate}
\end{defn}

\begin{rem}\label{rem:only_check_properness_once}
	In the definition of properness, we could have equivalently required the existence of \emph{some $y_1$} such that, for every $R\ge 0$, the set  $\{s \in S \mid d(y_1 \cdot s, y_1\cdot \dom{s})\le R\}$ is covered by finitely many $E(S)$-cosets. Indeed, let $z_1\in X_1$ be any other point, and let $D=d(y_1,z_1)$. If $d(z_1 \cdot s, z_1\cdot \dom{s})\le R$ for some $R\ge 0$, then 
    $$d(y_1 \cdot s, y_1\cdot \dom{s})\\
    \le d(y_1 \cdot s, z_1\cdot s)+d(z_1 \cdot s, z_1\cdot \dom{s})+d(z_1 \cdot \dom{s}, y_1\cdot \dom{s})\\
    \le R+2D,$$
    where we used that the action is by $1$-Lipschitz maps. This shows that 
    \[
		\{ s \in S \mid d(z_1 \cdot s, z_1\cdot \dom{s})\le R\}\subseteq \mathcal{C}(y_1,R+2D) E(S).
	\]
\end{rem}

\begin{lem}
    If $\mathcal{M}$ is a finite quasi-generating set for $S$, then the $S$-action on $(S,d_{\mathcal{M}})$ is proper and cobounded.
\end{lem}

\begin{proof}
Let $x_1=1$ be the identity element. Since $S=1\cdot S$, we immediately see that the action is cobounded. Moving to properness, by \Cref{rem:only_check_properness_once} we can restrict to $y_1=1$. Moreover, if $d_{\mathcal{M}}(s,\dom{s})\le R$, then $s=m_1\ldots m_k \dom{s}$ for some $m_i\in \mathcal{M}$ and $k\le R$ (here we are using that the Cayley distance can be computed only using non-idempotent elements, in view of \Cref{lem:cayley_dist_without_idempotents}). This shows that $s\in \mathcal{M}^{\le\lceil R\rceil}E(S)$, where $\lceil R\rceil$ is the smallest integer greater than $R$ and 
$$\mathcal{M}^{\le\lceil R\rceil}=\{m_1\ldots m_k\mid k\le \lceil R\rceil,\, m_i\in \mathcal M\mbox{ for all }i=1,\ldots,k\}.$$ Hence we can set $\mathcal{C}(1,R)=\mathcal{M}^{\le\lceil R\rceil}$.
\end{proof}

\noindent We are finally ready to state the main theorem of this paper:
\begin{thm}\label{thm:MilnorSchwarz}
	Let $S$ be an inverse monoid and $(X, E(S),p)$ be a presheaf of geodesic metric spaces. 
	If $S$ acts on $X$ properly and coboundedly then:
    \begin{enumerate}
        \item There exists a finite, quasi-generating set $\mathcal{M}\subseteq S$.
        \item\label{item:MS:qi} There exists $x_1\in X_1$ such that the orbit map $(S,d_{\mathcal{M}})\to X$ sending $s\in S$ to $x_1\cdot s$ is an order-preserving quasi-isometry.
    \end{enumerate}
\end{thm}

\begin{proof}
	By coboundedness there exist an 
    $x_1 \in X_1$ and $T\ge 0$ such that $B(x_1,T) \cdot S = X$. Let $$G \coloneq \{ s \in S \mid d(x_1 \cdot s, x_1\cdot \dom{s})\le 2T+1\},$$ we work to show that $G$ is a generating set of $S$. For every $s \in S$, $X_{\dom{s}}$ is a geodesic metric space, so there exists a geodesic
	$\gamma$ from $x_1 \cdot \dom{s}$ to $x_1 \cdot s$ in  $X_{\dom{s}}$.
    Take points along the geodesic at most a distance $1$ apart: concretely,
    let $$x_1 \cdot \dom{s} = p_1, \,p_2,\, \cdots,\, p_k = x_1 \cdot s$$
    be a series of points
    such that $d(p_i,p_{i+1}) \leq 1$ and $k \leq d(x_1 \cdot \dom{s}, x_1 \cdot s) + 2$.
    Since each $p_{i}$ is in $B(x_1,T) \cdot S$ there exists elements $s_{i}$ in $S$ 
	such that $d(p_{i}, x_1 \cdot s_{i}) \leq T$; for $i=k$ we can choose $s_k=s$. This implies that $x_1 \cdot s_i \in X_{\dom{s}}$ and hence
	$\dom{s_{i}} = \dom{s}$ for all $i$.
	Define $u_1 = s_1$ and $u_{i} = s_{i} s_{i-1}^{-1}$ for $2 \leq i \leq k$; notice that $u_k=ss_{k-1}^{-1}$. Then,
	\[
		s = s\dom{s}^{k-1}=s\dom{s_{k-1}}\ldots\dom{s_1}=(s s_{k-1}^{-1})(s_{k-1}s_{k-2}^{-1})\ldots (s_2s_1^{-1})s_1 = u_{k} \ldots u_{1}.
	\]
    Now, by assumption
    $d(x_1 \cdot s_1, x_1 \cdot \dom{s_1}) \leq T$, so $s_1\in G$. Moreover, for all $2 \leq i \leq k$,
	\[d(x_1 \cdot u_i, x_1 \cdot \dom{u_i}) = d(x_1 \cdot s_is_{i-1}^{-1}, x_1 \cdot s_{i-1}s_i^{-1}s_is_{i-1}^{-1})\]\[ =d(x_1 \cdot s_is_{i-1}^{-1}, x_1 \cdot s_{i-1}s_{i-1}^{-1}) \leq d(x_1\cdot  s_{i}, x_1 \cdot s_{i-1}) \leq 2T + 1,\]
	where we used that $s_i^{-1}s_i=\dom{s_i}=\dom{s_{i-1}}=s_{i-1}^{-1}s_{i-1}$. 
    Hence,  all $u_i$ are in $G$, so $G$ is a generating set.
    By properness $G \subseteq \mathcal{C}E(S)$ where $\mathcal{C}= \mathcal{C}(x_1, 2T+1)$, so
    $\mathcal{M} = \mathcal{C}$ is a finite quasi-generating set of $S$.

    We now work to prove (2). The orbit map $S\to x_1\cdot S$ is $T$-coarsely surjective by coboundedness, so we focus on showing it is a quasi-isometric embedding. Note that $$d_{\mathcal{M}}(s,t) < \infty \iff \dom{s}=\dom{t} \iff
    d(x_1 \cdot s, x_1 \cdot t) < \infty.$$ Hence, if $d_{\mathcal{M}}(s,t)<\infty$ then 
    \[d(x_1 \cdot s, x_1 \cdot t) =d(x_1 \cdot s\dom{t}, x_1 \cdot t\dom{s})= d(x_1 \cdot st^{-1}, x_1 \cdot \dom{st^{-1}}),\]
    where we used \Cref{lem:s_induces_isom} for the last equality. The same argument applies to the Cayley metric $d_\mathcal{M}$; hence it is enough to show that, for every $s\in S$, $d(x_1\cdot \dom{s}, x_1\cdot s)$ is linearly bounded above and below in terms of $d_{\mathcal{M}}(\dom{s}, s)$. 
    
   Write
    $s = u_{k} \cdots u_1$ as before. For each $i$ denote $v_{i} = u_{i} \cdots u_1 \dom{s}$, with $v_0 = \dom{s}$. By \Cref{rem:order_under_products}
\[
	\dom{s} = \dom{v_{k}} 
	\leq \dom{v_{k-1}} \leq \dom{v_{k-2}} \cdots\leq \dom{v_0} = \dom{s}
.\]
Hence, $\dom{v_i} = \dom{s}$ for all $i$, so each $u_i$ labels a (possibly loop) edge in the Schützenberger graph of $s$ connecting $v_{i-1}$ to $v_i$. Therefore, by construction of the $u_i$,
$$d_{\mathcal{M}}( \dom{s}, s) \leq k \leq d(x_1 \cdot \dom{s}, x_1 \cdot s) + 2.$$
For the other inequality, let $n=d_{\mathcal M}(s,\dom{s})$, and write $s=t_n\ldots t_1\dom{s}$, where each $t_i\in \mathcal M$ (this can be done by \Cref{lem:cayley_dist_without_idempotents}). Set $v_i=t_i\ldots t_1 \dom{s}$, with $v_0=\dom{s}$. Note that, for all $i$,
\[\dom{v_{i}} = \dom{v_{i+1}} = \dom{t_{i+1}v_{i}} = v_{i}^{-1}\dom{t_{i+1}}v_{i};\] hence
\[\ran{v_{i}}=v_i \dom{v_i}v_i^{-1}=\ran{v_i}\dom{t_{i+1}}\ran{v_i}= \ran{v_i}\dom{t_{i+1}}\leq \dom{t_{i+1}},\]
where we used that idempotents commute. Therefore, combining \Cref{lem:s_induces_isom} and the fact that the action is by $1$-Lipschitz maps, we get
 \[
	 d(x_1 \cdot v_{i+1}, x_1 \cdot v_{i}) =d(x_1 \cdot t_{i+1}v_i, x_1 \cdot v_{i}) = d(x_1 \cdot t_{i+1} \ran{v_{i}}
	 , x_1 \cdot \ran{v_{i}})
	 \leq d(x_1 \cdot t_{i+1}, x_1 \cdot \dom{t_{i+1}})
.\]
Hence
\[
	d(x_1 \cdot \dom{s}, x_1 \cdot s)
	\leq \sum_{i=1}^{n} d(x_1 \cdot v_{i-1}, x_1 \cdot v_{i})\leq n \max_{i=1,\ldots, n} \{d(x_1 \cdot t_{i}, x_1 \cdot \dom{t_{i}}) \}.\]
Since the $t_i$'s belong to the finite set $\mathcal M$, the quantity $d(x_1 \cdot t_{i}, x_1 \cdot \dom{t_{i}})$ is uniformly bounded, thus concluding the proof.
\end{proof}

\begin{rem}[Non-proper actions]\label{rem:no_properness_is_fine}
	Without the properness assumption, the above proof still produces a generating set $G$ such that the orbit map $(S,d_{G})\to X$ is a quasi-isometry. The only difference is that, for the converse inequality, one writes $s=t_n\ldots t_1$ where each $t_i$ belongs to $G$ instead of $\mathcal M$; then the uniform bound in the last line follows from the assumption that $d(x_1 \cdot t, x_1 \cdot \dom{t})\le 2T+1$ whenever $t\in G$.
\end{rem}

\begin{rem}\label{rem:nofinitegen}
    Notice that, in general, one cannot hope to get finite generation from an action on a presheaf, without further assumptions on the structure of the poset $E(S)$. Indeed, let $E$ be any uncountable meet-semilattice with a top element $1_e$, and let $G$ be any finitely generated group. Then $S=E\times G$ is an inverse monoid, with multiplication $(e,g)\cdot(f,h)=(ef,gh)$ and $E(S)=E$. Hence $S$ is quasi-finitely generated, but cannot be finitely generated as it is uncountable. Since $G$ can be any finitely generated group, and since the restriction maps $\{e\}\times G\to \{ef\}\times G$ are isometries, this example also shows that no geometric properties of the fibers, nor of the restriction maps of a presheaf, is enough to ensure finite generation.
\end{rem}

\begin{rem} \label{rem:need_1_for_qi} 
The assumption that $S$ is an inverse \emph{monoid} is necessary in \Cref{thm:MilnorSchwarz}. Indeed, the existence of an orbit map which is a quasi-isometry necessarily implies that the idempotents have a top element, and therefore that $S$ is an inverse monoid by \Cref{lem:unique_max_implies_monoid}. To see this, let $x\in X_e$ be any point, and suppose that the orbit map $(S, d_{\mathcal{M}}) \to X$ defined by $s \in S$ to $x \cdot s$ is a quasi-isometry. For every other idempotent $f\in E(S)$, $ef$ and $f$ both map to $x \cdot f$ under the orbit map, so $d_{\mathcal{M}}(f,ef)<+\infty$; since each Sch\"{u}tzenberger graph contains a unique idempotent, this means that $ef=f$, i.e. $f\le e$ in the partial order.
\end{rem}

\section{Comparison to work of Chung-Martínez-Szakács}
We present here an alternative proof of \Cref{thm:MilnorSchwarz}, building on work of Chung-Martínez-Szakács \cite{Chung_Martinez_szakacs} and a construction analogous to the Vietoris-Rips complex for groups acting on metric spaces.

\begin{defn}[Uniform discreteness]
     An extended metric space $(S,d)$ is \emph{uniformly discrete} if $\inf_{s\neq t}d(s,t)>0$. 
\end{defn}

\begin{defn}
    Let $S$ be an inverse semigroup, and let $d$ be an extended metric on $S$ whose components are the $\mathcal L$-classes. We say $d$ is:
    \begin{itemize}
        \item \emph{right-subinvariant} if $d(sx, tx) \le  d(s, t)$ for every $s, t, x \in S$;
        \item \emph{proper} if for every $r\ge 0$ there is some finite set $F_r\subseteq S$ such that, for every $x,y\in S$ with $x\neq y$ and $d(x, y) \le r$,  there exists $f\in F_r$ such that $y =fx$. 
        \item \emph{Uniformly proper} if there exists some finite set $F_1\subseteq S$ such that, for every $r\ge 0$ and every $x,y\in S$ with $x\neq y$ and $d(x, y) \le r$, there exists $f\in (F_1)^{\leq\lceil r\rceil}$ such that $y =fx$.
    \end{itemize} 
\end{defn}

\begin{thm}[{\cite[Theorem 1]{Chung_Martinez_szakacs}}]\label{thm:cms}
    Let $S$ be an inverse semigroup. Then the following statements are equivalent:
    \begin{enumerate}
        \item $S$ is quasi-countable; 
        \item $S$ admits a proper and right-subinvariant uniformly discrete extended metric whose components are the $\mathcal L$-classes.
    \end{enumerate}
Moreover, such a metric is unique up to bijective coarse equivalence.
\end{thm}

\noindent In the quasi-finitely generated case, one can promote properness to uniform properness:

\begin{prop}\label{prop:qfg_iff_unifproper}
        Let $S$ be an inverse semigroup. Then the following statements are equivalent:
    \begin{enumerate}
        \item\label{item:qfg} $S$ is quasi-finitely generated;
        \item\label{item:unifproper} $S$ admits a uniformly proper and right-subinvariant uniformly discrete extended metric whose components are the $\mathcal L$-classes.
    \end{enumerate}
\end{prop}

\begin{proof}
    \eqref{item:qfg}$\Rightarrow$\eqref{item:unifproper}: If $S$ is quasi-finitely generated, then the Cayley metric corresponding to the finite quasi-generating set satisfies the requirements, as pointed out in \cite[Example 3.4]{Chung_Martinez_szakacs}.

    \eqref{item:unifproper}$\Rightarrow$\eqref{item:qfg}: We argue exactly as in the proof of \cite[Theorem 3.23]{Chung_Martinez_szakacs}. Let $F_1$ be as in the definition of uniform properness, and let $s\in S-E(S)$. If we set $D\coloneq d(s,\dom{s})$, which is finite since the components of $d$ are the $\mathcal{L}$-classes, then $s\in (F_1)^{\leq\lceil D\rceil} \dom{s}$, proving that $F_1$ is a quasi-generating set.
\end{proof}

\begin{construction}[Vietoris-Rips graph]
    Let $S$ be an inverse monoid acting on a presheaf of geodesic metric spaces $(X, E(S),p)$, and let $x_1\in X_1$ be a basepoint. Given a constant $R\ge 0$, let $S^R_{x_1}$ be the simplicial graph whose vertex set is $S$ and where two $s,t\in S$ are adjacent if and only if $d(x_1\cdot s,x_1\cdot t)\le R $. We equip $S$ with the path metric $d^R_{x_1}$ where edges of $S^R_{x_1}$ have length $1$.
\end{construction}
\noindent The construction above does depend on the explicit choices of the basepoint $x_1$ and the constant $R$. However, as we shall see in \Cref{lem:dR_qi_X} below, in the presence of a cobounded action, any choice of the basepoint yields a simplicial graph which is quasi-isometric to $X$, provided that $R$ is large enough. 

\begin{prop}\label{prop:action_gives_metric}
    Let $S$ be an inverse monoid and $(X, E(S),p)$ be a (global) presheaf of geodesic metric spaces, and suppose that $S$ acts on $X$ properly and coboundedly. There exists $R_0 \geq 0$ and $x_1 \in X_1$ such that, for every $R\ge R_0$, $d^R_{x_1}$ is a uniformly discrete, proper, right-subinvariant extended metric whose components coincide with the $\mathcal L$-classes.
\end{prop}

\begin{proof}
Let $x_1 \in X_1$ be the element fulfilling the coboundedness condition. 
We check one requirement at a time.
\begin{itemize}
    \item \emph{Components}: Recall that, since $x_1\in X_1$ and the $S$-action on the base $E(S)$ of the presheaf is by conjugation, $(s,t)\in\mathcal{L}$ if and only if $x_1\cdot s$ and $x_1\cdot t$ lie in the same fiber $X_e$. Now, if $s$ and $t$ are within finite distance in $S^R_{x_1}$, then clearly $x_1\cdot s$ and $x_1\cdot t$ lie at finite distance in $X$, and therefore belong to the same fiber $X_e$. Conversely, suppose $x_1\cdot s$ and $x_1\cdot t$ lie in the same fiber. As in the proof of \Cref{thm:MilnorSchwarz}, if $T$ is the coboundedness constant from \Cref{defn:proper_and_cobounded}, one can use any geodesic between $x_1\cdot s$ and $x_1\cdot t$ to construct a finite sequence $s=s_0,\ldots, s_k=t$ such that $d(x_1\cdot s_i, x_1\cdot s_{i+1})\leq R_0\coloneq 2T+1$ for all $i$; therefore, if $R\ge R_0$, $s$ and $t$ lie in the same component of $S^R_{x_1}$.

    \item \emph{Uniform discreteness}: this is simply because $S^R_{x_1}$ is a graph with edges of length $1$, so any two distinct points lie at distance at least $1$.
    \item \emph{Uniform properness}: The proof will resemble the arguments from \Cref{thm:MilnorSchwarz}. Since $d^R_{x_1}$ is a path metric, it is enough to prove that there exists a finite subset $F_1\subseteq S$ such that, if $d^R_{x_1}(s,t)\le 1$, then $s\in F_1 t$. 
    \\
    We claim that we can actually set $F_1=\mathcal{C}(R,x_1)$, as in the definition of properness. Indeed, by definition $d(x_1\cdot s, x_1\cdot t)\le R$, and therefore $d(x_1\cdot st^{-1}, x_1\cdot tt^{-1})\le R$. Notice that 
    \[
    \dom{st^{-1}}=ts^{-1}st^{-1}=t\dom{s}t^{-1}=t\dom{t}t^{-1}=tt^{-1},
    \]
    where we used that $\dom{s}=\dom{t}$ as their $d^R_{x_1}$-distance is finite. Hence $r\coloneq st^{-1}$ satisfies $d(x_1\cdot r, x_1\cdot \dom{r})\le R$, so $r\in F_1 E(S)$. Summing up what we did so far, we got $$s=s\dom{s}=s\dom{t}=st^{-1}t=rt=met,$$ where $m\in F_1$ and $e\in E(S)$. 
    \\
    Finally, from $s=met$ we get that $\dom{s}\le \dom{et}\le \dom{t}$ by \Cref{rem:order_under_products}, and since $\dom{s}=\dom{t}$ we get that $\dom{et}=\dom{t}$. In turn 
    $$t=t\dom{t}=t\dom{et}=tt^{-1}et=et.$$
    This shows that $s=met=mt$, so $s\in F_1 t$, as required. 
    \item \emph{Right-subinvariance}: Since $S$ acts on $X$ on the right by $1$-Lipschitz maps, the action preserves pairs of points within distance $R$, and therefore maps edges of $S^R_{x_1}$ to edges.\qedhere 
\end{itemize}
\end{proof}

\noindent As a consequence of \Cref{prop:qfg_iff_unifproper}, we immediately get:
\begin{cor}
    Any $S$ as in \Cref{prop:action_gives_metric} is quasi-finitely generated.
\end{cor}

\noindent Furthermore, the coarse equivalence from \Cref{thm:cms} can be improved to a quasi-isometry:
\begin{cor}
    The metric $d^R_{x_1}$ is quasi-isometric to the Cayley metric with respect to any finite quasi-generating set.
\end{cor}

\begin{proof}
    By \Cref{thm:cms}, the two metrics are coarsely equivalent. Moreover, Cayley metrics with respect to finite quasi-generating sets and path metrics on graphs are both quasi-geodesic, and coarsely equivalent quasi-geodesic (extended) metric spaces are actually quasi-isometric: see e.g. \cite[Section 0.2.D]{gromov_asymp_inv}.
\end{proof}

\noindent Hence our \Cref{thm:MilnorSchwarz} follows from the next observation:

\begin{lem}\label{lem:dR_qi_X}
     There exists $T\ge 0$ such that, if $R\ge T$, then $(S,d^R_{x_1})$ is quasi-isometric to $X$.
\end{lem}

\begin{proof}
Consider the map $\phi\colon S\to X$ mapping $s$ to $x_1\cdot s$. This map is $R$-Lipschitz since, if $s,t\in S$ are adjacent in $S^R_{x_1}$ then $d(x_1\cdot s, x_1\cdot t)\le R$ by construction. Furthermore $d^R_{x_1}(s,t)\le d(x_1\cdot s, x_1\cdot t)/R+1$, since any geodesic $\gamma$ connecting $x_1\cdot s$ to $x_1\cdot t$ can be subdivided into at most $d(x_1\cdot s, x_1\cdot t)/R+1$ subsegments of length at most $R$. Finally, coarse surjectivity follows from coboundedness: if $R$ is greater than the coboundedness constant $T$, then for every $y\in X$ there exists $s\in S$ such that $y\in B(x_1,T)\cdot s$, and therefore $d(y,x_1\cdot s)\le R$ as the action is by $1$-Lipschitz maps. 
\end{proof}

\bibliography{biblio}
\bibliographystyle{alpha}
\end{document}